\documentclass[11pt,reqno]{amsart}

\usepackage[a4paper,left=35mm,right=35mm,top=30mm,bottom=30mm,marginpar=25mm]{geometry}

\usepackage{amsmath}
\usepackage{amssymb}
\usepackage{amsthm}
\usepackage{eurosym}
\usepackage[dvips]{graphics}
\usepackage{graphicx}
\usepackage{epsfig}
\usepackage{hyperref}
\usepackage{dsfont}
\usepackage[displaymath,mathlines]{lineno}

\allowdisplaybreaks

\usepackage{hyperref}

\usepackage{ifthen}
\usepackage{ esint }
\makeindex

\newcommand{\arcsinh}{\operatorname{arcsinh}}

\newcommand{\Rr}{{\mathbb{R}}}

\newcommand{\Tt}{{\mathbb{T}}}

\newcommand{\p}{\partial}

\newcommand{\epsi}{\varepsilon}

\def\leq{\leqslant}
\def\geq{\geqslant}

\numberwithin{equation}{section}

\newtheoremstyle{thmlemcorr}{10pt}{10pt}{\itshape}{}{\bfseries}{.}{10pt}{{\thmname{#1}\thmnumber{
#2}\thmnote{ (#3)}}}
\newtheoremstyle{thmlemcorr*}{10pt}{10pt}{\itshape}{}{\bfseries}{.}\newline{{\thmname{#1}\thmnumber{
#2}\thmnote{ (#3)}}}
\newtheoremstyle{defi}{10pt}{10pt}{\itshape}{}{\bfseries}{.}{10pt}{{\thmname{#1}\thmnumber{
#2}\thmnote{ (#3)}}}
\newtheoremstyle{remexample}{10pt}{10pt}{}{}{\bfseries}{.}{10pt}{{\thmname{#1}\thmnumber{
#2}\thmnote{ (#3)}}}
\newtheoremstyle{ass}{10pt}{10pt}{}{}{\bfseries}{.}{10pt}{{\thmname{#1}\thmnumber{
A#2}\thmnote{ (#3)}}}

\theoremstyle{thmlemcorr}
\newtheorem{theorem}{Theorem}
\numberwithin{theorem}{section}
\newtheorem{lemma}[theorem]{Lemma}

\theoremstyle{thmlemcorr*}
\newtheorem{theorem*}{Theorem}
\newtheorem{lemma*}[theorem]{Lemma}
\newtheorem{corollary*}[theorem]{Corollary}
\newtheorem{proposition*}[theorem]{Proposition}
\newtheorem{problem*}[theorem]{Problem}
\newtheorem{conjecture*}[theorem]{Conjecture}

\theoremstyle{defi}

\theoremstyle{remexample}
\newtheorem{remark}[theorem]{Remark}

\newtheorem{lem}[theorem]{Lemma}
\newtheorem{pro}[theorem]{Proposition}

\theoremstyle{ass}

\usepackage{color}

\begin{document}

\title{One-dimensional forward-forward mean-field games}
 
\author{Diogo A.
        Gomes}         \address[D. A. Gomes]{
        King Abdullah University of Science and Technology (KAUST), CEMSE Division , Thuwal 23955-6900. Saudi Arabia.}

\author{Levon Nurbekyan}

\address[Levon Nurbekyan]{
        King Abdullah University of Science and Technology (KAUST), CEMSE
        Division , Thuwal 23955-6900. Saudi Arabia.}

\author{Marc Sedjro}

\address[Marc Sedjro]{
        King Abdullah University of Science and Technology (KAUST), CEMSE
        Division , Thuwal 23955-6900. Saudi Arabia.}

\subjclass[2010]{} 

\thanks{
        The authors were supported by KAUST baseline and start-up funds. 
}
\date{\today}

\begin{abstract}
While the general theory for the terminal-initial value problem for mean-field games (MFGs) has achieved a substantial progress, 
the corresponding  forward-forward problem is still poorly understood --
 even in the one-dimensional setting. Here, we consider one-dimensional forward-forward MFGs, study the existence of solutions and their long-time convergence. 
First, we discuss the relation between these models and systems of conservation laws. 
In particular, we identify new conserved quantities and study some qualitative properties of these
systems. 
Next, we introduce a class of wave-like equations
that are equivalent to forward-forward MFGs, and we derive a novel formulation as a system of conservation laws.
For first-order logarithmic forward-forward MFG, we establish the existence of a global solution. Then, we consider a class of explicit solutions and show the existence of shocks.
Finally, we examine parabolic forward-forward MFGs and establish the long-time convergence of the solutions. 
\end{abstract}

\maketitle

\section{Introduction}

Mean-field games (MFGs) are models for large populations of competing rational agents who  seek to optimize an individual objective function. A typical model is the  {\em backward-forward MFG}. In one dimension, this game  is determined by following the  system of partial differential equations (PDEs):
\begin{equation}\label{eq: General MFG}
 \begin{cases}
 -u_t + H(u_x)= \varepsilon u_{xx}   + g(m), \\
 m_t-(H'( u_x) m)_x=\varepsilon  m_{xx} .   \end{cases}
\end{equation}
For convenience, the spatial domain, corresponding to the variable $x$,  is the  $1$-dimensional torus, $ \Tt$, identified with the interval $[0,1]$.  The time domain, corresponding to the variable  $t$, is the interval $[0,T]$ for some terminal time, $T>0$. The unknowns in the above system are  $u:\Tt\times [0,T]\to \Rr$ and $m:\Tt\times [0,T]\to \Rr$. In this game, each agent seeks to solve an optimal control problem. The function $u(x,t)$ is the value function 
for this control problem for an agent located at  $x\in \Tt$ at the time $t$. 
This control problem is determined by a  
Hamiltonian, $H: \Rr\to \Rr$, $H\in C^2$, and a coupling between each agent and the mean field, $m$, given by the function $g: \Rr^+\to
\Rr$,  $g\in C^1$.
The first equation in \eqref{eq: General MFG} is a Hamilton-Jacobi equation
and expresses the optimality of the value function, $u$. 
 For each $t\in [0, T],$  $m$ is a probability density in $\Tt$.  The second equation of (\ref{eq: General MFG}), the Fokker-Planck equation, determines the evolution of $m$. The parameter $\varepsilon\geq 0$ is the viscosity coefficient in the Fokker-Planck equation;  $\varepsilon=0$ corresponds to  {\em first-order MFGs} and  $\varepsilon>0$ to {\em parabolic MFGs}. The system  \eqref{eq: General MFG} is endowed with terminal-initial conditions; the initial value of $m$ is prescribed at $t=0$ and the terminal value of $u$, at $t=T$:
 \begin{equation}
 \label{itc}
 \begin{cases}
 u(x,T) = u_T(x) \\
 m(x,0) = m_0(x).
  \end{cases}
  \end{equation}
As a result, \eqref{eq: General MFG}-\eqref{itc} is called the {\em terminal-initial value problem} or the {\em backward-forward MFG}. 

Here,  we examine a related model, the {\em forward-forward MFG} problem. This model  is constructed by the reversal of the time variable in the Hamilton-Jacobi equation in (\ref{eq: General MFG}).  Accordingly,  the   {\em forward-forward MFG system} in $\mathbb{ T}\times
[0,T]$ is determined by  
\begin{equation}
\label{ffmfg}
 \begin{cases}
  u_t + H(u_x)= \varepsilon u_{xx}   + g(m)\\
 m_t-(H'( u_x) m)_x=\varepsilon  m_{xx}, 
 \end{cases}
  \end{equation}
together with the  {\em initial-initial condition}:  
\begin{equation}\label{ini-ini}
 \begin{cases}
 u(x,0) = u_0(x) \\
 m(x,0) = m_0 (x).
  \end{cases}
\end{equation}
The forward-forward model was  introduced  in \cite{DY}
to  approximate
stationary MFGs.
The key insight is that the parabolicity in (\ref{ffmfg}) should imply
the long-time convergence to a stationary solution.
In the preceding MFG, a typical Hamiltonian, $H$, is the quadratic Hamiltonian,  $H(p)=\frac{p^{2}}{2}$, or for $\gamma>1,$ the power-like Hamiltonian, $H(p)=\frac{1}{\gamma}|p|^\gamma$ or  $H(p)=(1+p^2)^{\frac \gamma 2}$. Regarding the  coupling nonlinearity, $g$, here, we consider the power-like case, $g(m)=m^\alpha$ for some $\alpha>0$, or the logarithmic case, $g(m)=\ln m$.

Considerable research has focused on proving the existence of solutions for backward-forward MFGs. For example, weak solutions for parabolic problems were considered in  \cite{ ll2, porretta2}, strong solutions for parabolic problems in  \cite{GPim2, GPim1, ll2}, and weak solutions for first-order MFGs in \cite{Cd2, GraCard}.  The stationary case was also investigated in detail since it was first considered in  \cite{ll1}. For this case, the existence of classical and weak solutions was investigated in  \cite{GMit,GP,GPat,GPatVrt}. The uniqueness of solution is well understood (both for stationary and time-dependent MFGs) via the monotonicity method introduced in \cite{ll1,ll2}. 
Monotonicity properties are also fundamental for the existence theory developed in \cite{FG2}. One-dimensional MFGs provide examples and guidance 
for the study of higher-dimensional problems and numerical methods
\cite{AFG}.  Moreover, these games have an independent interest  in problems in  networks and graphs \cite{ CaCa2016, camillinet, MR3146865} and congestion  \cite{GP, GLP2}.

In contrast to that of the backward-forward case, our understanding of forward-forward MFGs is limited.
In particular,  the existence and the long-time convergence of the forward-forward model have not been addressed, except in a few cases, see  \cite{GPff}
and 
 \cite{llg2}. In 
\cite{llg2}, the  forward-forward problem was examined in the context of eductive stability of stationary MFGs with a logarithmic coupling.  In \cite{GPff}, the existence and regularity
of solutions for the forward-forward, uniformly
parabolic MFGs with subquadratic Hamiltonians was proven. Except for these cases, the question of existence and regularity is open  in all other regimes.  In the case of forward-forward MFGs without viscosity, these questions are particularly challenging. Moreover, the long-time convergence has not been established even in the parabolic case. Nevertheless, numerical results in  \cite{AcCiMa} and \cite{Ci} indicate that convergence holds and that the forward-forward model approximates well stationary solutions.

Not only as an effective tool to approximate stationary problems, the forward-forward MFGs can also be regarded  as a learning game. In backward-forward MFGs, the density of the agents is transported by the (future) optimal trajectories of an optimal control problem. In the  forward-forward model, the interpretation the evolution of the agents is less straightforward. In this model, the density is transported by past optimal trajectories because the corresponding control problem has initial data, not terminal data. Thus, 
the  actions of the agents are determined by a learning strategy where past densities drive their evolution. 

This paper is structured as follows.  In Section \ref{cl}, we reformulate
\eqref{eq: General MFG}-\eqref{itc} and  (\ref{ffmfg})-(\ref{ini-ini}) as systems of conservation laws. There, we identify new conserved quantities for these problems in the case where $\varepsilon=0$. Conserved quantities are fundamental in analyzing PDEs and in testing and validating numerical methods. 
Here, they are used in the long-time convergence analysis. 
Next, in Section \ref{wte}, we derive wave-type equations that are equivalent to (\ref{ffmfg})-(\ref{ini-ini}). For example, for  the first-order,  logarithmic forward-forward model, we   obtain the PDE
\[
u_{tt}=(1+u_x^2) u_{xx}.
\] 
The preceding equation is
 equivalent to an elastodynamics problem. 
The corresponding elastodynamics equations have entropy solutions when the stress function is monotone.
Thus, we obtain the existence of solutions for the original MFG.
 In addition, using results from \cite{BEJ}, we identify a class of explicit solutions for the logarithmic MFGs. These explicit solutions provide an example where shocks arise in the forward-forward model. 
 Finally, in Section \ref{pmfg}, we examine forward-forward parabolic MFGs. Here, the entropies 
identified in Section \ref{cl} play an essential role in our analysis of the long-time behavior of solutions.  Due to the parabolicity, these entropies are dissipated and force the long-time convergence of the solutions of (\ref{ffmfg})-(\ref{ini-ini}).

\section{Systems of conservation laws and first-order MFGs}
\label{cl}

Here, we consider deterministic MFGs; that is, $\varepsilon=0.$  In this case,  \eqref{eq:
General MFG} and \eqref{ffmfg} are equivalent to conservation laws, at least for smooth enough solutions. In this preliminary section, we examine these conservation laws and identify  conserved quantities. In Section \ref{pmfg}, we use these conserved quantities to establish the long-time convergence of the parabolic forward-forward MFG  \eqref{ffmfg}.

Before proceeding, we recall some well-known results on systems  
conservation
laws in one dimension. 
We consider a conservation law of the form \begin{equation}\label{eq: Conservation laws}
 U_t + (F(U))_x=0,
 \end{equation}
 where $U: \mathbb{R}\times \mathbb{T} \longrightarrow \mathbb{R}^2$
is the unknown and $F :\mathbb{R}^2 \longrightarrow \mathbb{R}^2  $ is the
flux function. We say that $(E,Q)$ is an entropy/entropy-flux pair if 
  \begin{equation}\label{eq : entropy-entropy flux}
  (E(U))_t + (Q(U))_x=0
  \end{equation}
 for any smooth solution of \eqref{eq: Conservation
laws}. We note that  (\ref{eq : entropy-entropy flux}) implies that  $E(U)$ is
a conserved quantity  if the solution $U$ of (\ref{eq: Conservation
laws}) is smooth; that is, 
  \begin{equation}
 \dfrac{d}{dt} \int_{\mathbb{T}} E (U) dx =- \int_{\mathbb{T}} (Q(U))_x dx
 =0.
  \end{equation}

\subsection{Backward-Forward MFG}\label{sec: bfmfg}

 Now, we assume that \eqref{eq:
General MFG} has a smooth enough solution, for example, $u,m\in C^2(\Tt\times (0,\infty))\cap C(\Tt\times
[0,\infty)).$  We set $v=u_x$ and differentiate the first equation in \eqref{eq:
General MFG} 
with respect to $x$.   Accordingly, we obtain the following system
\begin{equation}\label{eq: foba mfg system2}
\begin{cases}
v_t+(g(m)-H(v))_x=0,\\
m_t-(mH'(v))_x=0.
\end{cases}
\end{equation}

To investigate the existence of an entropy for (\ref{eq: foba mfg system2}), we look for an entropy/entropy-flux  $(E,Q)$ satisfying \eqref{eq : entropy-entropy flux} for $U=(v,m)$.
By expanding \eqref{eq
: entropy-entropy flux},  we get 
 
\begin{equation}\label{eq : fb mfg entropy 01}
 \frac{\p E}{\p v} v_t+\frac{\p E}{\p m} m_t  +  \frac{\p Q}{\p v} v_x+\frac{\p Q}{\p m} m_x =0.
  \end{equation}
 In light of (\ref{eq: foba mfg system2}),   (\ref{eq : fb mfg entropy 01}) becomes

\begin{equation}\label{eq : fb mfg entropy 0}
 \frac{\p E}{\p v} H'(v) v_x- \frac{\p E}{\p v} g'(m) m_x + \frac{\p E}{\p m} H'(v) m_x  +\frac{\p E}{\p m} m H''(v)v_x   +  \frac{\p Q}{\p v} v_x+\frac{\p
Q}{\p m} m_x =0.
  \end{equation}
Thus, 
\begin{equation}\label{eq : fb mfg entropy 1}
 \frac{\p Q}{\p v}= - \frac{\p E}{\p v} H'(v)-\frac{\p E}{\p m} m H''(v)\qquad\hbox{and}\qquad \frac{\p Q}{\p m}=\frac{\p E}{\p v} g'(m) -\frac{\p E}{\p
m} H'(v).
\end{equation}
Consequently, we obtain the following PDE for $E$
\begin{equation}\label{eq : fb mfg entropy 1}
 \frac{\p }{\p m}\left(  - \frac{\p E}{\p v} H'(v)-\frac{\p E}{\p m} m H''(v)\right) =
\frac{\p }{\p v}\left( \frac{\p E}{\p v} g'(m) -\frac{\p E}{\p
m} H'(v)\right). 
\end{equation}
After elementary computations,
the above equation becomes \begin{equation}\label{eq: cons E condition}
\frac{1}{ H''(v)}
 \frac{\p^2 E}{\p v^2}+\frac{1}{P''(m)} \frac{\p^2 E}{\p m^2}=0,
\end{equation}
where 
\begin{equation}
\label{Pofm}
P''(m)=\frac{g'(m)}{m}.
\end{equation}
The preceding equation has the following trivial solutions:
 $$E(v,m)=\alpha v+\beta m,\ \alpha,\beta \in \Rr.$$ By inspection, we can verify that  the following two expressions solve \eqref{eq: cons E condition}:
\[
E(v,m)=mv \quad \text{and}  \quad E(v,m)=H(v)-P(m).
\]
Moreover, if $g$ is increasing, $P$ is a convex function whereas if $g$ is decreasing, $P$ is concave. 

Using separation of variables and writing
\[
E=\Phi(v)\Psi(m),
\]
we derive the following conditions
\[
\begin{cases}
\frac{1}{H''(v)}\frac{\Phi''(v)}{\Phi(v)}=\lambda\\
\frac{1}{P''(m)}\frac{\Psi''(m)}{\Psi(m)}=-\lambda.
\end{cases}
\]

The conditions above take a simple form when $g(m)=\dfrac{m^2}{2}$, which corresponds to $P(m)=\frac{m^2}{2}$, and $H(v)=\frac{v^2}{2}$, namely
\[
\begin{cases}
\Phi''(v)=\lambda\Phi(v)\\
\Psi''(m)=-\lambda \Psi(m) .
\end{cases}
\]
Thus,  we have solutions of the form $\Phi(v)=e^{\pm \sqrt{\lambda}v}$ and $\Psi(m)=e^{\pm i\sqrt{\lambda}m}$, which have exponential growth or oscillation depending upon the sign of $\lambda$. 
In addition to these conservation laws, there are also polynomial conservation laws. For illustration, some of these are shown in Table \ref{T1}. In Table \ref{T2}, we present some conservation laws for the anti-monotone backward-forward MFG with $g(m)=-\frac{m^2}{2}$.
These laws are 
straightforward to compute as the determining equations for $E$ are
\[
\frac{\partial^2 E}{\partial m^2}+\frac{\partial^2 E}{\partial v^2}=0 
\]
in the monotone case and
\[
\frac{\partial^2 E}{\partial m^2}-\frac{\partial^2 E}{\partial v^2}=0 
\]
in the anti-monotone case. 
In both cases, these equations have solutions that are homogeneous polynomials in $m$ and $v$. In the monotone case, these conservation laws are the real and imaginary parts of $(m+i v)^k$. 
In the anti-monotone case, some of the conservation laws are coercive and, thus,  control the $L^p$ norms of $v$\ and $m$ (at least for smooth solutions).

\begin{table}
\centering
\begin{tabular}{|c|c|}\hline
Degree &$E(v,m)$\\\hline
3&$ v^3-3 m^2 v$ \\\hline
3& $m^3-3 m v^2$ \\\hline
4&$-6 m^2 v^2+m^4+v^4$\\\hline
4&$ m v^3-m^3 v$\\\hline
5&$ -10 m^2 v^3+5 m^4 v+v^5$\\\hline
5&$ -10 m^3 v^2+5 m v^4+m^5$\\\hline
6&$ 15 m^4 v^2-15 m^2 v^4-m^6+v^6$\\\hline
6&$ m^5 v-\frac{10}{3} m^3 v^3+mv^5$\\\hline
\end{tabular}
\smallskip
\caption{Conservation laws for the backward-forward MFG with $H(v)=\frac{v^2}{2}$ and $g(m)=\frac{m^2}{2}$ up to degree 6. }
\label{T1}
\end{table}

\begin{table}
\centering
\begin{tabular}{|c|c|}\hline
Degree &$E(v,m)$\\\hline
3&$3 m^2 v+v^3$\\\hline
3&$3 m v^2+m^3$\\\hline
4&$6 m^2 v^2+m^4+v^4$\\\hline
4&$ m^3 v+m v^3$\\\hline
5& $10 m^2 v^3+5
m^4 v+v^5$\\\hline
5& $10 m^3 v^2+5 m v^4+m^5$\\\hline
6&
$ 15 m^4 
v^2+15 m^2 v^4+m^6+v^6$\\\hline
6& $m^5 v+\frac{10}{3} 
m^3 v^3+m v^5$\\\hline
\end{tabular}
\smallskip
\caption{Conservation laws for the backward-forward MFG with $H(v)=\frac{v^2}{2}$
and $g(m)=-\frac{m^2}{2}$ up to degree 6. }
\label{T2}
\end{table}

\subsection{Forward-forward MFG}\label{ss: ffmfg}
As previously,  we assume that \eqref{ffmfg} has a solution, $u,m\in C^2(\Tt\times
(0,\infty))\cap C(\Tt\times
[0,\infty))$, and we set $v:=u_x$. We differentiate the first equation in \eqref{ffmfg} with respect to $x$ and obtain the system:
\begin{equation}\label{eq: foff mfg system2}
 \begin{cases}
 v_t-(g(m)-H(v))_x=0,\\
 m_t-(mH'(v))_x=0.
 \end{cases}
 \end{equation}
We begin by examining the entropies for \eqref{eq: foff mfg system2}; that is, we look for $(E,Q)$ satisfying \eqref{eq : entropy-entropy flux} for $U=(v,m)$. We expand (\ref{eq : entropy-entropy flux}) to get 
 \begin{equation}\label{XYZ}
 \frac{\p E}{\p v} v_t+\frac{\p E}{\p m} m_t  +  \frac{\p Q}{\p v} v_x+\frac{\p
Q}{\p m} m_x =0.
  \end{equation}
 In light of (\ref{eq: foff mfg system2}), (\ref{XYZ}) becomes
\begin{equation}\label{eq : ff mfg entropy 0}
- \frac{\p E}{\p v} H'(v) v_x+ \frac{\p E}{\p v} g'(m) m_x +\frac{\p E}{\p
m} H'(v) m_x  +\frac{\p E}{\p m} m H''(v)v_x   +  \frac{\p Q}{\p v} v_x+\frac{\p
Q}{\p m} m_x =0.
  \end{equation}
Thus, 
\begin{equation}\label{eq : ff mfg entropy 1}
 \frac{\p Q}{\p v}=  \frac{\p E}{\p v} H'(v)-\frac{\p E}{\p m} m H''(v)\qquad\hbox{and}\qquad
\frac{\p Q}{\p m}=-\frac{\p E}{\p v} g'(m) -\frac{\p E}{\p m} H'(v).
\end{equation}
Consequently,
\begin{equation}\label{eq : ff mfg entropy 1}
 \frac{\p }{\p m}\left( \frac{\p E}{\p v} H'(v)-\frac{\p E}{\p m} m H''(v)\right) =
\frac{\p }{\p v}\left(-\frac{\p E}{\p v} g'(m) -\frac{\p E}{\p
m} H'(v)\right). 
\end{equation}
 This last equation simplifies to
 \begin{equation}\label{eq: cons E condition 1}
 \frac{1}{H''(v)} \frac{\p^2 E}{\p v^2}+\frac{2 H'(v)}{H''(v)g'(m)}
 \frac{\p^2 E}{\p v \p m}-\frac{m}{g'(m)}  \frac{\p^2 E}{\p m^2}=0.
 \end{equation}
 The preceding equation has a trivial family of solutions,
 $$E(v,m)=\alpha v+\beta m,\ \alpha,\beta \in \Rr.$$ 
Moreover, \eqref{eq: cons E condition 1} admits a solution of the form:
 $$E(v,m)=H(v)+P(m)$$
 with $P(m)$ as in \eqref{Pofm}.
In contrast with the backward-forward case, here, if $g$ is increasing, the previous entropy is convex. This observation is crucial for our proof of convergence of the forward-forward mean-field games with viscosity. For illustration, we consider the case $H(v)=\frac {v^2}{2}$.
In Tables \ref{T3} and \ref{T4}, we present some polynomial conservation laws for, respectively, a monotone, $g(m)=\frac{m^2}{2}$, and an anti-monotone, $g(m)=-\frac{m^2}{2}$, quadratic forward-forward MFG. These conservation laws satisfy
\[
\frac{\partial^2 E}{\partial v^2}\pm\frac{2 v}{m}\frac{\partial^2E}{\partial v\partial m}\mp\frac{\partial^2 E}{\partial m^2}=0,  
\]
where the $-$ sign corresponds to the monotone case and the $+$ sign to the anti-monotone case.

\begin{table}
\centering
\begin{tabular}{|c|c|}\hline
Degree &$E(v,m)$\\\hline
3&$v^3-3 m^2 v$ \\\hline
4& $-2 m^2 v^2-\frac{1}{3} m^4+v^4$ \\\hline
4&$m^3 v$\\\hline
5&$  -2 m^2 v^3-3 m^4 v+v^5$\\\hline
6&$\frac{45}{7} 
m^4 v^2-\frac{15}{7} m^2 v^4+\frac{3 m^6}{7}+v^6$\\\hline
\end{tabular}
\smallskip
\caption{Conservation laws for the forward-forward MFG with $H(v)=\frac{v^2}{2}$
and $g(m)=\frac{m^2}{2}$ up to degree 6. }
\label{T3}
\end{table}

\begin{table}
\centering
\begin{tabular}{|c|c|}\hline
Degree &$E(v,m)$\\\hline
3&$ 3 m^2 v+v^3$ \\\hline
4& $2 m^2 v^2-\frac{1}{3} m^4+v^4$ \\\hline
4&$m^3 v$\\\hline
5&$ 2 m^2 v^3-3 m^4 v+v^5$\\\hline
6&$\frac{45}{7} 
m^4 v^2+\frac{15}{7} m^2 v^4-\frac{3}{7} m^6+v^6$\\\hline
\end{tabular}
\smallskip
\caption{Conservation laws for the forward-forward MFG with $H(v)=\frac{v^2}{2}$
and $g(m)=-\frac{m^2}{2}$ up to degree 6. }
\label{T4}
\end{table}

\section{Wave-type equations}
\label{wte}

Here, we introduce a class of wave-type equations that are equivalent to forward-forward MFGs. Using these equations, we rewrite the forward-forward MFG as a new  system of conservation laws.
For  $g(m)=m^\alpha$, this new system depends polynomially in $\alpha$ in contrast with \eqref{eq: foff mfg system2} where the dependence on  $\alpha$ is exponential. This new formulation is of interest for the numerical simulation of forward-forward MFGs with a large value $\alpha$ and substantially simplifies the computation of conserved quantities.  
Subsequently, we consider the logarithmic nonlinearity and, using a result from DiPerna, we prove the existence of a global solution for the forward-forward problem. Moreover, this solution is bounded in $L^\infty$. Finally, also for the logarithmic nonlinearity, we investigate the connection between this new formulation and a class of equations introduced in \cite{BEJ}. In particular, we provide a representation formula for some solutions of the forward-forward MFG and establish the existence of shocks.  

\subsection{Wave equations and forward-forward MFGs}We continue our study of forward-forward MFGs by reformulating \eqref{ffmfg}  as a  scalar nonlinear wave equation.
Here, 
we assume that $H,g$ are smooth and $g$ is either strictly increasing or decreasing; that is, $g'\neq 0$. From the first equation in \eqref{ffmfg}, we have that
 \begin{equation}\label{eq: density in terms of vfunction}
        m=g^{-1}(u_t+H(u_x)).
  \end{equation}       
We differentiate $\eqref{eq: density in terms of vfunction}$ with respect to $t$ and $x$ to obtain, respectively, 
\begin{equation}
 \begin{aligned}\label{eq: diff1 density in terms of vfunction}
 m_t=(g^{-1})'\left(u_t+H(u_x)\right) (u_{tt}+H'(u_x)u_{xt})
 \end{aligned}
 \end{equation}
 and
 \begin{equation}\label{eq: diff2 density in terms of vfunction}
 \begin{aligned}
 (m H'(u_x))_x&=(g^{-1})'(u_t+H(u_x))(u_{tx}+H'(u_x)u_{xx})H'(u_x)\\
 &+g^{-1} (u_t+H(u_x)) H''(u_x)u_{xx}.\\
 \end{aligned}
 \end{equation}
Next, we combine (\ref{eq: diff1 density in terms of vfunction}) and (\ref{eq: diff2 density in terms of vfunction}) and get
 \begin{align*}        
 m_t-(mH'(u_x))_x&=(g^{-1})'(u_t+H(u_x))(u_{tt}+H'(u_x)u_{xt})\\
 &-(g^{-1})'(u_t+H(u_x))(u_{tx}+H'(u_x)u_{xx})H'(u_x)\\
 &-g^{-1}(u_t+H(u_x)) H''(u_x)u_{xx}.
 \end{align*}
 Hence, the second equation in \eqref{eq: foff mfg system2}  yields
 \begin{align*}             
 (g^{-1})'(u_t+H(u_x))\left( u_{tt}-(H'(u_x))^2u_{xx}\right) =g^{-1}(u_t+H(u_x)) H''(u_x)u_{xx},
 \end{align*}
 or, equivalently,
 \begin{align}
 u_{tt}=\left(
 (H'(u_x))^2
 +
 g'(g^{-1}(u_t+H(u_x)))
 g^{-1}(u_t+H(u_x))
 H''(u_x)
 \right)u_{xx};
 \end{align}
 that is,
 \begin{equation}\label{eq:   we foff mfg}
 u_{tt}=\left((H'(u_x))^2+mg'(m) H''(u_x)\right)u_{xx}.
 \end{equation}
 Thus,  \eqref{eq: foff mfg system2}   is equivalent to the nonlinear second-order equation \eqref{eq: we foff mfg} coupled with \eqref{eq: density in terms of vfunction}. Moreover, if $g$ is increasing, the preceding equation is hyperbolic. In the particular case where $g(m)=
 \ln m$, \eqref{eq: we foff mfg} takes the simpler form
        \begin{equation}\label{eq: wave equation with H}
        u_{tt}=\left((H'(u_x))^2+H''(u_x)\right)u_{xx}.
        \end{equation}       
 
 \subsection{A new system of conservation laws}
 
 Now, we consider the wave equations introduced in the preceding section and reformulate them as a new system of conservation laws. 
For that, we set $v=u_x$ and $w=u_t$. Then, \eqref{eq:   we foff mfg} is equivalent to
 \begin{equation}\label{eq: foff mfg system1}
 \begin{cases}
 v_t=w_x,\\
 w_t=\left((H'(v))^2+g'( g^{-1}(w+H(v))) g^{-1}(w+H(v)) H''(v)\right)v_x.
 \end{cases}
 \end{equation} 
 We set \[\phi(v,w)=(H'(v))^2+g'( g^{-1}(w+H(v))) g^{-1}(w+H(v)) H''(v).
 \]
  Accordingly, \eqref{eq: foff mfg system1} becomes
   \begin{equation}\label{eq: foff mfg system1 phi}
 \begin{cases}
 v_t=w_x,\\
 w_t=\phi(v,w)v_x.
 \end{cases}
 \end{equation} 
 
 In the sequel, we choose 
 \begin{equation} \label{eq : Stored energy}
 H(v)=\dfrac{v^2}{2}\qquad\hbox{and}\qquad g(m)=m^{\alpha}.
 \end{equation}
  Consequently, we have that 
$$mg'(m)=\alpha g(m).$$
 Therefore,  \eqref{eq: foff mfg system1 phi} takes the form
 \begin{equation*}
 \begin{cases}
 v_t=w_x,\\
 w_t=\left(v^2+\alpha(w+ v^2 )\right)v_x.
 \end{cases}
 \end{equation*} 
 Next, we search for a conserved quantity, $F(v,w),$ for the preceding system.
Arguing as before, we  see that $F$ is conserved
if and only if
 \begin{equation}\label{eq: cons F condition}
 \frac{\p^2 F}{\p v^2}=\frac{\p}{\p w}\left(\frac{\p F}{\p w}\\ \phi(v,w)\right),
 \end{equation}
 where $\phi(v,w)=v^2+\alpha(w+ v^2 )$.
 A particular solution of \eqref{eq: cons F condition} is 
 \begin{equation}\label{eq: particular conserved quantity 1}
F(v,w)= w+ \dfrac{\alpha}{2}v^2.
 \end{equation} 
Accordingly, we set 
\begin{equation}\label{eq: particular conserved quantity 2}
z(x,t)= w(x,t)+ \dfrac{\alpha}{2}v^2(x,t).
\end{equation} 
Thus, we have  that
\begin{align*}
        z_t &= w_t+ \alpha v v_t\\
            &=\left(v^2+\alpha\left(w+ \dfrac{v^2}{2} \right)\right)v_x+\alpha v w_x\\
            &= \left(1+\dfrac{\alpha}{2}\right)v^2v_x +\alpha v_x w+\alpha v w_x\\
            &=\left( \dfrac{1}{3}+\dfrac{\alpha}{6}\right)v^3_x +\alpha (vw)_x\\
            &=\left( \left( \dfrac{1}{3}+\dfrac{\alpha}{6}\right)v^3 +\alpha vw\right) _x.
\end{align*}
Hence, we obtain the following equivalent system of conservation laws
 \begin{equation}\label{eq: foff mfg system3}
 \begin{cases}
 z_t=\left( \left( \frac{1}{3}+\frac{\alpha}{6}-\frac{\alpha^2}2\right)v^3 +\alpha v z\right) _x,\\
 v_t=\left(z+\dfrac{\alpha}{2}v^2\right)_x.
 \end{cases}
 \end{equation}
We observe that $\alpha$ is no longer in the exponent of the foregoing equation. Therefore, the growth of the nonlinearity becomes polynomial with a fixed degree for any exponent
$\alpha$. This property is relevant for the numerical analysis and simulation of these games. Moreover, in this formulation, we obtain further polynomial conservation laws for \eqref{eq:
foff mfg system3} shown in Table \ref{T5}. 
 
\begin{table}
\centering
\begin{tabular}{|c|c|}\hline
Degree &$E(z,v)$\\\hline
2&$ v z$ \\\hline
4& $3 \alpha ^2 v ^4-\alpha  v ^4-12 \alpha  v ^2 z-2 v ^4-12 z^2$ \\\hline
5&$v  \left(9 \alpha ^2 v ^4-3 \alpha  v ^4-20 \alpha  v ^2 z-6 v ^4-60 z ^2\right)$\\\hline
6&$ 6 \alpha ^3 v ^6-2 \alpha ^2 v ^6-4 \alpha  v ^6+5 \alpha ^2 v ^4 z-5 \alpha  v ^4 z-60 \alpha  v ^2 z ^2-10 v  ^4 z-20 z ^3$\\\hline
\end{tabular}
\smallskip
\caption{Conservation laws for the modified forward-forward MFG \eqref{eq: foff mfg system3}
up to degree 6. }
\label{T5}
\end{table}

\subsection{Forward-forward MFGs with a logarithmic nonlinearity -- existence of a solution}

Here, we prove the existence of a solution of \eqref{eq: wave equation with H} for a  quadratic Hamiltonian. For our proof, we use the ideas in the preceding subsection and rewrite    \eqref{eq: wave equation with
H} as a system of conservation laws. The system we consider here is a special case of the ones   investigated in \cite{DiPerna83}, in the whole space, and in \cite{DeStTz00}, in the periodic case. More precisely, we examine the system
\begin{equation} \label{eq : elastodynamics}
\begin{cases}
v_t -w_x= 0 \\
w_t - \sigma(v)_x=0
\end{cases}
\end{equation}
with the initial conditions
\begin{equation}
\begin{cases}
v(x,0)=v_0(x)\\
w(x,0)=w_0(x).
\end{cases}
\end{equation}
 Here, $\sigma : \mathbb{R}\to \mathbb{R}$ is a $C^2$ function, $\sigma'> 0$, $(v,w)$ is the unknown
and $(x,t)\in \Tt\times[0,T]$. We consider initial data $v_0, w_0\in L^\infty(\Tt)$. As pointed out in \cite{DeStTz00},  if (\ref{eq
: elastodynamics}) has a $C^1$ solution then  there exists $u$ such that
$w=u_t,$ $v=u_x$  and a straightforward computation yields \begin{equation} \label{eq
: Wave elastodynamics}
u_{tt}- (\sigma(u_x))_x=0.
\end{equation}
 In addition, for a quadratic Hamiltonian, $H(p)=\frac{p^2}{2}$,
 \eqref{eq : Wave elastodynamics} is equivalent to \eqref{eq: wave equation
with H} for  
\begin{equation}
\label{sigma}
\sigma(z)=z+\frac{z^3}{3}.
\end{equation} 
By proving the existence of a solution to  
\eqref{eq: wave equation with H}, we get a solution of  the corresponding forward-forward MFG. 

In  \cite{DiPerna83}, the author considers the   viscosity approximation\begin{equation} \label{eq : viscosity elastodynamics}
\begin{cases}
v_t^\varepsilon -w_x^\varepsilon= \varepsilon v _{xx} \\
v_t^\varepsilon - \sigma(u^\varepsilon)_x=\varepsilon w_{xx}
\end{cases}
\end{equation}
and proves that,  in the limit $\varepsilon\to 0$, $(u^\varepsilon, v^\varepsilon)$ converges to a solution of 
 \eqref{eq : elastodynamics}. 
For the reader convenience, we reproduce a result from \cite{DeStTz00} that ensures the existence of a solution of (\ref{eq : elastodynamics}) in $\mathbb{\Tt}\times[0,T]$. 

\begin{theorem}
Let $\sigma$ be given by \eqref{sigma}. Suppose that $v_0, w_0\in L^\infty(\Tt).$
Then \eqref{eq : elastodynamics} has a weak solution $v,w\in L^\infty(\Tt\times [0,T])$.               
\end{theorem}
\begin{proof}
The theorem follows from the results in  \cite{DeStTz00}
because    $\sigma'>0$ and $\sigma''$ vanishes at a single point.
 Furthermore, as shown in  \cite{DiPerna83}, because
\begin{equation}
z\sigma''(z)>0 \qquad  \forall z\neq 0\label{eq: stress constraints 2}
\end{equation}
  and the initial data belongs to $L^{\infty}(\mathbb{T}\times[0,T])$, the theory of invariant regions developed in \cite{ChCoSm77} ensures that  $$\|v\|_{L^{\infty}(\mathbb{R}\times[0,T])}+ \|w\|_{L^{\infty}(\mathbb{R}\times[0,T])}\leq C. $$
 \end{proof}

\subsection{Logarithmic forward-forward MFGs and Hamilton-Jacobi flows}

We end this section with a brief discussion of the connection between 
the logarithmic forward-forward MFG and a class of Hamilton-Jacobi
flows introduced in  \cite{BEJ}. As in discussed in that reference, we consider the Hamilton-Jacobi equation
 \begin{equation}\label{eq : Hamilton Jacobi}
u_t + G(u_x)=0.
  \end{equation}
Assuming smoothness in the equation, we differentiate respectively with respect to $x$ and $t$ to obtain:
 \begin{equation}\label{eq : Hamilton Jacobi diff x}
u_{tx} + G'(u_x)u_{xx}=0
  \end{equation}
and 
 \begin{equation}\label{eq : Hamilton Jacobi diff t}
u_{tt} + G'(u_x)u_{tx}=0.
  \end{equation}
Next, we combine (\ref{eq : Hamilton Jacobi diff x}) and (\ref{eq : Hamilton Jacobi diff t}) to get 
 \begin{equation}\label{eq : Hamilton Jacobi diff t,x combined}
u_{tt} - [G'(u_x)]^2 u_{xx}=0.
  \end{equation}
Finally,  we set 
\begin{equation}\label{eq : H for HJ technique}
 G(p)=\begin{cases}
\frac{1}{2}[p\sqrt{1+p^2}+\arcsinh(p)]\qquad &p\geq 0\\
-\frac{1}{2}[p\sqrt{1+p^2}+\arcsinh(p)]\qquad &p<0,
  \end{cases}
\end{equation}
so that (\ref{eq : Hamilton Jacobi diff t,x combined}) becomes
\[
u_{tt}-(1+u_x^2)u_{xx}=0. 
\]
We observe that G is convex. Thus, we can compute the solution  of \eqref{eq : Hamilton Jacobi} by using the Lax-Hopf formula.  For that, 
we introduce the Legendre transform 
\[
G^*(v)=\sup_{p} pv -G(p) 
\]
and, according to  
 the Lax-Hopf formula, we get the following representation for the solution of \eqref{eq : Hamilton Jacobi}
\begin{equation}
\label{efor}
u(x,t)=\inf_{y} t G^*\left(\frac{x-y}{t}\right)+u(y,0).
\end{equation}
If $u(x,0)$ is differentiable, so is $u(x,t)$ for $0<t<T^*$, where $T^*$\ is the time of the first shock. 

Now, we set
\[
m=e^{H(u_x(x,t))-G(u_x(x,t))}. 
\]
Then, for smooth enough solutions, a simple calculation gives
\[
m_t-(m u_x)_x=0.
\]
Thus, we see that $u$ and $m$ solve the forward-forward MFG
\begin{equation}
\label{lmfg}
\begin{cases}
  u_t + \frac{u_x^2}{2}= \ln m\\
 m_t-( m u_x)_x=0.  
 \end{cases}
\end{equation}
Finally, because \eqref{eq : Hamilton Jacobi diff t,x combined} depends only on the $G'(u_x)^2$, we can repeat the discussion above for the equation 
\[
u_t-G(u_x)=0, 
\]
and obtain another explicit solution. 

The examples we discuss in this section show that \eqref{lmfg} develops shocks in finite time as the regularity of $u$ is at best the regularity of the solutions of the Hamilton-Jacobi equation \eqref{eq : Hamilton Jacobi}.
Moreover, the convergence results for Hamilton-Jacobi equations (see, for example,  \cite{ CGMT,MR2237158, FATH4,  MR2396521, MR1457088})\ show that
the function $u$ given by \eqref{efor} converges (up to additive constants) as $t\to \infty$ to a stationary solution of
\[
G(u_x)=\overline{G}.
\]

\section{Parabolic MFGs}
\label{pmfg}

In Section \ref{ss: ffmfg}, we examined the first-order forward-forward MFGs ($\epsi=0$) and determined several conserved quantities (entropies). In the
parabolic ($\epsi>0$) case, these entropies are dissipated.
 Here, we use this dissipation to establish the long-time convergence of solutions.

As before, by differentiating 
 \eqref{ffmfg} with respect to $x$, we get
 \begin{equation}\label{eq: parabolic ffmfg}
        \begin{cases}
        v_t+(g(m)-H(v))_x=\epsi v_{xx},\\
        m_t-(mH'(v))_x= \epsi m_{xx},
        \end{cases}
\end{equation}
where $v=u_x$.  We assume that $g$ is $C^1$ and strictly increasing, and that $H$ is $C^2$ and strictly convex; that is, $H''(v)>0$ for all $v\in \Rr$.  Additionally, we impose
\begin{equation}\label{eq: fixedmeanvalues}
        \int\limits_{\Tt} v(x,0)dx=0, \quad \int\limits_{\Tt} m(x,0)dx=1.
\end{equation}
The foregoing conditions are natural because $v$ is the derivative of a periodic function, $u,$ and $m$ is a probability density.
A straightforward computation yields the following result. 
\begin{lem}
Suppose that $v,m \in C^2(\Tt \times (0,+\infty))\cap C(\Tt \times (0,+\infty))$ solve  \eqref{eq: parabolic ffmfg}. Furthermore, let $E(v,m)$ be a $C^2$ entropy for  \eqref{eq: foff mfg system2}; that is, $E(v,m)$ satisfies \eqref{eq: cons E condition 1}. Then, \begin{equation}\label{eq: time derivative of entropy}
\frac{d}{dt}\int_{\mathbb{T}} E(v, m) dx= -\varepsilon \int_{\mathbb{T}} (v_x, m_x)^T D^2E(v, m)(v_x, m_x)dx.
\end{equation}
\end{lem}
Now, let $P(m)$ be as in \eqref{Pofm}. Note that $P$ is strictly convex when $g$ is strictly increasing. 
\begin{lem}
        Let $\epsi>0$. Suppose $v,m \in  C^2(\Tt \times (0,+\infty))\cap C(\Tt \times (0,+\infty))$ solve  \eqref{eq: parabolic ffmfg} and satisfy \eqref{eq: fixedmeanvalues}. Then, for all $t\geq 0$, we have that        \begin{equation}\label{eq: fixedmeanvalues t}
        \int\limits_{\Tt} v(x,t)dx=0, \quad \int\limits_{\Tt} m(x,t)dx=1.
        \end{equation}
        Furthermore, if $g$ is increasing, we have that
        \begin{align}\label{eq: time derivative of convex solutions}
        &\frac{d}{dt}\int_{\Tt} H(v(x,t))+P(m(x,t)) dx\\&\qquad = -\epsi \int_{\Tt} H''(v(x,t))v_x^2(x,t)+P''(m(x,t))m_x^2(x,t)dx\leq 0.
        \end{align}
\end{lem}
\begin{proof}
In Section \ref{ss: ffmfg}, we observed that $E_0(v,m)=v, E_1(v,m)=m,$ and $E_2(v,m):=H(v)+P(m)$ are entropies for \eqref{eq: foff mfg system2}. Hence, we apply \eqref{eq: time derivative of entropy} to $E_0,E_1$ and $E_2$ and obtain \eqref{eq: fixedmeanvalues t} and \eqref{eq: time derivative of convex solutions}.
The inequality in \eqref{eq: time derivative of convex solutions} follows from the convexity of $H$ and $P$.
\end{proof}

\subsection{Poincar\'{e}-type inequality}

To establish the long-time convergence, we need the following Poincar\'{e}-type inequality:
\begin{theorem}\label{thm: poincare}
        Let $I\subset \Rr$ be an open interval and $\Phi \in C^2(I)$ be a strictly convex function. Furthermore, let $\Psi \in C^1(I)$ be such that
        \begin{equation}\label{eq: psi}
                \Psi'(s)=\sqrt{\Phi''(s)},\quad s\in I.
        \end{equation}
        Then, for every $f:\Tt\to I$,   $f \in C^1(\Tt)$, we have
                \begin{equation}\label{ineq: poincare 1}
                \int\limits_{\Tt} \Phi(f(x))dx - \Phi\left(\int\limits_{\Tt} f(x)dx\right)\leq C_{\Phi}(a,b) \int\limits_{\Tt} \Phi''(f(x))f'(x)^2dx,
                \end{equation}
                where $a=\min \limits_{\Tt}f$, $b=\max \limits_{\Tt}f,$ and                 \begin{equation}\label{eq: cab}
                        C_{\Phi}(a,b)=\frac{\Phi(a)+\Phi(b)-2\Phi\left(\frac{a+b}{2}\right)}{(\Psi(b)-\Psi(a))^2}.
                \end{equation}
             Moreover, if
                \begin{equation}\label{eq: c}
                        C_{\Phi}=\sup \limits_{a,b \in I} C_{\Phi}(a,b)<\infty,
                \end{equation}
                then
                \begin{equation}\label{ineq: poincare 2}
                \int\limits_{\Tt} \Phi(f(x))dx - \Phi\left(\int\limits_{\Tt} f(x)dx\right)\leq C_{\Phi} \int\limits_{\Tt} \Phi''(f(x))f'(x)^2dx
                \end{equation}
                for all $f:\Tt\to I$,  $f\in C^1(\Tt)$.
\end{theorem}
\begin{proof}
        Because \eqref{ineq: poincare 2} is an immediate consequence of \eqref{ineq: poincare 1},  we only need to prove the latter inequality. For that, next, we show that for every $f:\Tt\to I$,  $f\in C^1(\Tt)$, such that $a=\min \limits_{\Tt}f$ and $b=\max \limits_{\Tt} f$, we have         \begin{equation}\label{ineq: reversejensen}
                \int\limits_{\Tt} \Phi(f(x))dx - \Phi\left(\int\limits_{\Tt} f(x)dx\right)\leq \Phi(a)+\Phi(b)-2\Phi\left(\frac{a+b}{2}\right)
        \end{equation}
        and
        \begin{equation}\label{ineq: nondegenracy}
                \int\limits_{\Tt} \Phi''(f(x))f'(x)^2dx \geq (\Psi(b)-\Psi(a))^2.
        \end{equation}
        If  $a=b,$ $f$ is constant and the result is trivial.  Thus,  we assume $a<b$. Let $A=\int\limits_{\Tt} f(x)dx$. We have that $a\leq A\leq b$. Furthermore, because $\Phi$ is convex, we have that
        \[\Phi(s)\leq \frac{b-s}{b-a} \Phi(a)+\frac{s-a}{b-a}\Phi(b)=:L(s),\quad \forall\ s\in[a,b]. 
        \]
Now, we observe that $\Phi(s)-L(s)$ is a convex function that vanishes at $s=a,b$. Accordingly, for $a<s<\frac{a+b}{2}$ there exists $\lambda>\frac 1 2$ such that 
\[
\frac{a+b}{2}=\lambda s +(1-\lambda) b.
\] 
Therefore, 
\begin{align}
\label{mi}
\Phi\left(\frac{a+b}{2}\right)-L\left(\frac{a+b}{2}\right)&\leq  \lambda (\Phi(s)-L(s))+(1-\lambda) (\Phi(b)-L(b))\\
&=\lambda (\Phi(s)-L(s)).\notag\end{align}
Arguing in a similar way for   $\frac{a+b}{2}\leq s<b$, we see that \eqref{mi} also holds for some $\lambda>\frac 1 2$.       
Consequently, we have

\begin{align*}
L(s)-\Phi(s)&\leq 2\left(L\left(\frac{a+b}{2}\right)-\Phi\left(\frac{a+b}{2}\right)\right)\\&=\Phi(a)+\Phi(b)-2\Phi\left(\frac{a+b}{2}\right)
\end{align*}
for all $s\in [a, b]$. Hence, we get
\[
\int\limits_{\Tt} \Phi(f(x))dx \leq \int\limits_{\Tt} L(f(x))dx=L\left(\int\limits_{\Tt}f(x)dx\right)=L(A).
\]
        Therefore,         \[\int\limits_{\Tt} \Phi(f(x))dx - \Phi\left(\int\limits_{\Tt} f(x)dx\right)\leq L(A)-\Phi(A)\leq \Phi(a)+\Phi(b)-2\Phi\left(\frac{a+b}{2}\right).
        \]
        Suppose $f(x_0)=a$ and $f(x_1)=b$. Then, we have that
        \begin{align*}
                \int\limits_{\Tt} \Phi''(f(x))f'(x)^2dx&=\int\limits_{\Tt} \left(\frac{d\Psi(f(x))}{dx}\right)^2dx\geq \left(\int\limits_{\Tt} \left|\frac{d\Psi(f(x))}{dx}\right|dx\right)^2\\
                &\geq \left(\int_{x_0}^{x_1} \left|\frac{d\Psi(f(x))}{dx}\right|dx\right)^2\geq \left|\int_{x_0}^{x_1} \frac{d\Psi(f(x))}{dx}dx\right|^2\\
                &=(\Psi(b)-\Psi(a))^2.
        \end{align*}
\end{proof}
Next, we present some  convex functions $\Phi$ for which \eqref{eq: c} holds.
\begin{pro}
        Let $I$ and $\Phi \in C^2(I)$ be one of the following:
        \begin{enumerate}
                \item $I=(0,\infty),\ \Phi(s)=s^p$, where $p>1$.
                \item $I=(0,\infty),\ \Phi(s)=s^p$, where $p<0$.
                \item $I=(0,\infty),\ \Phi(s)=-s^p$, where $0<p<1$.
                \item $I=(0,\infty),\ \Phi(s)=-\ln s$.
                \item $I=(0,\infty),\ \Phi(s)=s\ln s$.
                \item $I=\Rr,\ \Phi(s)=s^{2n}$, where $n\in \mathbb{N}$.
                \item $I=\Rr,\ \Phi(s)=e^{\alpha s}$, where $\alpha \in \Rr$.
        \end{enumerate}
        Then, $C_{\Phi}$ defined in \eqref{eq: c} is finite. Consequently,  \eqref{ineq: poincare 2} holds.
\end{pro}
\begin{proof}
The proof of the preceding result is elementary though tedious, and we omit it here. 
\end{proof}

\subsection{Stability of Jensen's inequality} The proof of the long-time convergence of the solutions of \eqref{eq: parabolic ffmfg} is based on the following stability property of Jensen's inequality: 
\begin{theorem}\label{thm: jensenstability}
Let $I\subset \Rr$ be an open interval, not necessarily bounded, and $\Phi \in C(I)$ a strictly convex function. Furthermore, let $A \in I$  and  $f_t:\Tt\to I$, $\{f_t\}_{t>0} \subset C(\Tt)$, be such that,
for all $t\geq 0$, \[\int\limits_{\Tt} f_t(x)dx=A
\]
and
\[\lim\limits_{t \to \infty} \int\limits_{\Tt} \Phi(f_t(x))dx-\Phi(A) =0.
\]
Then,\[\lim\limits_{t\to \infty} \int\limits_{\Tt}|f_t(x)-A|dx=0.
\]
\end{theorem}
\begin{remark}
        Note that we do not impose uniform $L^{\infty}$ bounds on the family $\{f_t\}_{t>0}$.
\end{remark}
Before proving Theorem \ref{thm: jensenstability}, we need the following technical lemma.
We recall that $\mathcal{L}^1$ denotes the one-dimensional Lebesgue measure. 

\begin{lemma}
Let $I\subset \Rr$ be some interval and $\Phi \in C(I)$ a convex function. Then, for every $f\in C(I),$ we have that
\begin{equation}\label{eq: JensenStability}
\int\limits_{\Tt} \Phi(f(x))dx-\Phi\left(\int\limits_{\Tt} f(x) dx\right) \geq p \Phi(A_1)+ q \Phi(A_2)-(p+q) \Phi(A)\geq 0,
\end{equation}
where
\begin{equation}
\label{eq: jensenstability_suppl}
A=\int\limits_{\Tt} f(x) dx,\quad p=\mathcal{L}^1(\{f< A\}),\quad q=\mathcal{L}^1(\{f\geq  A\}),
\end{equation}
\begin{equation}
\label{eq: jensenstability_supplB}
A_1=\fint\limits_{f<A} f(x)dx= A- \frac{\gamma(f)}{p},\quad A_2= \fint\limits_{f\geq A} f(x)dx= A+ \frac{\gamma(f)}{q},
\end{equation}
and
\begin{equation}
\label{eq: jensenstability_supplC}
\gamma(f)=\int\limits_{\Tt} (f(x)-A)^- dx=\int\limits_{\Tt} (f(x)-A)^+ dx=\frac{1}{2}\int\limits_{\Tt} |f(x)-A| dx.
\end{equation}
\end{lemma}
\begin{proof}
By rearranging \eqref{eq: JensenStability} and observing that $p+q=1$,  we get the  inequality
\begin{align*}
        &\int\limits_{f<A} \Phi(f(x))dx+\int\limits_{f\geq A} \Phi(f(x))dx\\
        &\geq \mathcal{L}^1(\{f(x)< A\}) \Phi\left(\fint\limits_{f<A} f(x)dx\right)+\mathcal{L}^1(\{f(x)\geq  A\}) \Phi\left(\fint\limits_{f\geq A} f(x)dx\right).
\end{align*}
The result follows by observing that the preceding inequality is a consequence of  Jensen's inequality.
\end{proof}
Now, we are ready to prove Theorem \ref{thm: jensenstability}.
\begin{proof}[Proof of Theorem \eqref{thm: jensenstability}] Let $p_t,q_t,A_1^t,A_2^t$ and $\gamma_t:=\gamma(f_t)$ be as in \eqref{eq: jensenstability_suppl}-\eqref{eq: jensenstability_supplC} for $f=f_t$. From \eqref{eq: JensenStability}, we have that
\[p_t \Phi(A^t_1)+ q_t \Phi(A^t_2)-(p_t+q_t) \Phi(A) \to 0.
\]
By contradiction, we assume that $f_t$ does not converge to the common average value $A$. Then, without loss of generality,  we can assume that
 there exists $\varepsilon_0>0$
such that \[\gamma_{t_n} \geq \epsi_0>0
\]
for some sequence $t_n\to\infty$. Consequently,
\[|A^{t_n}_1-A|=\frac{\gamma_{t_n}}{p_{t_n}}\geq \epsi_0
\]
and
\[|A^{t_n}_2-A|=\frac{\gamma_{t_n}}{q_{t_n}}\geq \epsi_0.
\]
Because $\Phi$ is strictly convex, we have that
\begin{align*}
        k=\inf\limits_{|s-A| \geq \epsi_0} \frac{\Phi(s)-\Phi(A)-(s-A) \alpha}{|s-A|} = \min\limits_{|s-A| = \epsi_0} \frac{\Phi(s)-\Phi(A)-(s-A)\cdot \alpha}{|s-A|} >0
\end{align*}
for any $\alpha$ in the subdifferential $\partial^-\Phi(A)$. Therefore,
\begin{align*}
p_{t_n} \Phi(A^{t_n}_1)+ q_{t_n} \Phi(A^{t_n}_2)-(p_{t_n}+q_{t_n}) \Phi(A)&=p_{t_n}(\Phi(A^{t_n}_1)-\Phi(A)-(A^{t_n}_1-A) \alpha)\\
&+q_{t_n}(\Phi(A^{t_n}_2)-\Phi(A)-(A^{t_n}_2-A) \alpha) \\
&\geq k p_{t_n} |A^{t_n}_1-A|+k q_{t_n} |A^{t_n}_2-A|\\
&=k \gamma_{t_n}\geq k \epsi_0,
\end{align*}
which is a contradiction.
\end{proof}
If we have uniform $L^{\infty}$ bounds,  we have the following stronger stability property for Jensen's  inequality:
\begin{theorem}\label{thm: jensenstability strong}
Let $I\subset \Rr$ be some interval and $\Phi \in C(I)$ a strictly convex function. Furthermore, let $a<b$ be real numbers and consider a family of functions $f_t:\Tt\to I,$ $\{f_t\}_{t>0}\subset C(\Tt),$ such that
\[ a\leq f_t(x) \leq b,\quad\forall x \in \Tt,\quad\forall t>0,
\]
and
\[\lim\limits_{t \to \infty} \int\limits_{\Tt} \Phi(f_t(x))dx-\Phi\left(\int\limits_{\Tt} f_t(x) dx\right) =0.
\]
Then, we have that
\begin{equation}\label{eq: strong stability l1}
        \lim\limits_{t\to \infty} \int\limits_{\Tt}|f_t(x)-A_t|dx=0,
\end{equation}
where $A_t=\int\limits_{\Tt} f_t(x)dx$. Consequently,
\begin{equation}\label{eq: strong stability lp}
        \lim\limits_{t\to \infty} \int\limits_{\Tt}|f_t(x)-A_t|^pdx=0
\end{equation}
for all $p>1$.
\end{theorem}
\begin{proof} Because $f_t$ is bounded, \eqref{eq: strong stability lp} follows from \eqref{eq: strong stability l1}. Therefore,  we only need to prove the latter. Let $p_t,q_t,A_1^t,A_2^t$ and $\gamma_t:=\gamma(f_t)$ be as in \eqref{eq: jensenstability_suppl}-\eqref{eq:
jensenstability_supplC}  for $f=f_t$.
By contradiction, we assume that there exists $\varepsilon_0>0$ such that 
\[\gamma_{t_n} \geq \epsi_0>0,
\]
for some $t_n\to \infty$. Accordingly, \[|A^{t_n}_1-A_{t_n}|=\frac{\gamma_{t_n}}{p_{t_n}}\geq \epsi_0,
\]
and
\[|A^{t_n}_2-A_{t_n}|=\frac{\gamma_{t_n}}{q_{t_n}}\geq \epsi_0.
\]
We have that $a\leq A_t,A^t_1,A^t_2\leq b. $ Therefore,  by compactness, we can assume that
\[A^{t_n}_1\to A_1,\quad A^{t_n}_2\to A_2,\quad A_{t_n}\to A,\quad p_{t_n}\to p,\quad q_{t_n} \to q, 
\]
extracting a subsequence if necessary. Moreover, we have that
\[|A_1-A|,|A_2-A| \geq \epsi_0>0.
\]
Furthermore, since $\Phi$ is continuous, we have that
\[p_{t_n} \Phi(A^{t_n}_1)+ q_{t_n} \Phi(A^{t_n}_2)-(p_{t_n}+q_{t_n}) \Phi(A_{t_n})\to p \Phi(A_1)+q\Phi(A_2)-(p+q)\Phi(A)=0, 
\]
using  \eqref{eq: JensenStability}. Note that
\[p_tA^t_1+q_tA^t_2=(p_t+q_t)A_t\]
for all $t>0$. Hence, 
\[pA_1+qA_2=(p+q)A.\]
Next, since $\Phi$ is strictly convex, we get that $p=0$ or $q=0$. But then $p_{t_n} \to 0$ or $q_{t_n} \to 0$. Suppose $p_{t_n} \to 0$. Then,
\[\epsi_0\leq \gamma_{t_n}= \int\limits_{f_{t_n}<A_{t_n}}|f_{t_n}(x)-A_{t_n}|dx\leq (b-a) \mathcal{L}^1(\{f_{t_n}<A_{t_n}\})=(b-a)p_{t_n},
\]
which is a contradiction. Similarly, we get a contradiction if $q_{t_n} \to 0$.
\end{proof}

\subsection{Parabolic forward-forward MFGs -- convergence}

Now, we are ready to prove the convergence result for  \eqref{eq: parabolic ffmfg}.
\begin{theorem}
        Let $H\in C^2(\Rr)$ be strictly convex and $g\in C^1\left((0,\infty)\right)$ be strictly increasing. Suppose that $C_{H},C_{P}<\infty$ (see \eqref{eq: c}), where $P$ is as in \eqref{Pofm}. Furthermore, let $v,m \in C^2(\Tt \times (0,+\infty))\cap C(\Tt \times
[0,+\infty)),\ m>0,$ solve \eqref{eq: parabolic ffmfg} and satisfy \eqref{eq: fixedmeanvalues}. Then, we have that
        \begin{equation}\label{eq: longtime l1}
                \lim \limits_{t\to \infty} \int\limits_{\Tt} |v(x,t)|dx=0,\quad \lim \limits_{t\to \infty} \int\limits_{\Tt} |m(x,t)-1|dx=0.
        \end{equation}
        Moreover, if         \[\sup\limits_{t\geq 0} \|v(\cdot,t)\|_{C(\Tt)} \quad \text{and}\quad \sup\limits_{t\geq 0} \|m(\cdot,t)\|_{C(\Tt)}<\infty,
        \]
        then, for all $1<p<\infty$,
        \begin{equation}\label{eq: longtime lp}
        \lim \limits_{t\to \infty} \int\limits_{\Tt} |v(x,t)|^pdx=0\quad\text{and} \quad  \lim \limits_{t\to \infty} \int\limits_{\Tt} |m(x,t)-1|^pdx=0.
        \end{equation}
\end{theorem}
\begin{proof}
        Let $C_0:=\max\{C_H,C_P\}$. Let
        \[I(t)=\int\limits_{\Tt} H(v(x,t))+P(m(x,t)) dx-H(0)-P(1).
        \]
        From \eqref{eq: fixedmeanvalues t},\ \eqref{eq: time derivative of convex solutions}, and \eqref{ineq: poincare 2}, we have that
        \begin{align*}
                \frac{dI(t)}{dt}&=-\epsi \int\limits_{\Tt} H''(v(x,t))v_x^2(x,t)+P''(m(x,t))m_x^2(x,t)dx\\
                &\leq - \frac{\epsi}{C_0} \left(\int\limits_{\Tt} H(v(x,t))dx-H(0)+\int\limits_{\Tt}P(m(x,t)) dx-P(1)\right)\\
                &=- \frac\epsi{C_0} I(t). 
        \end{align*}
        Therefore, we get
        \[I(t)\leq e^{-\frac\epsi{C_0} t} I(0)\qquad \forall t\geq 0,
        \]
        which yields
        \[\lim\limits_{t\to \infty}I(t)=0.
        \]
        Furthermore, by Jensen's inequality, we have that
        \[\int\limits_{\Tt} H(v(x,t))dx-H(0)\leq I(t),
        \]
        and
        \[\int\limits_{\Tt} P(m(x,t))dx-P(1)\leq I(t).
        \]
        Therefore, we get
        \[\lim\limits_{t\to \infty}\int\limits_{\Tt} H(v(x,t))dx-H(0)=\lim\limits_{t\to \infty}\int\limits_{\Tt} P(m(x,t))dx-P(1)=0,
        \]
        and we conclude using Theorem \ref{thm: jensenstability}.
\end{proof}

\def\polhk#1{\setbox0=\hbox{#1}{\ooalign{\hidewidth
			\lower1.5ex\hbox{`}\hidewidth\crcr\unhbox0}}} \def\cprime{$'$}


\begin{thebibliography}{10}
	
	\bibitem{DY}
	Y.~Achdou and I.~Capuzzo-Dolcetta.
	\newblock Mean field games: numerical methods.
	\newblock {\em SIAM J. Numer. Anal.}, 48(3):1136--1162, 2010.
	
	\bibitem{AcCiMa}
	Y.~Achdou, M.~Cirant, and M.~Bardi.
	\newblock Mean-field games models of segregation.
	\newblock {\em Preprint}, 2016.
	
	\bibitem{AFG}
	N.~Al-Mulla, R.~Ferreira, and D.~Gomes.
	\newblock Two numerical approaches to stationary mean-field games.
	\newblock {\em Preprint}.
	
	\bibitem{BEJ}
	E.~N. Barron, L.~C. Evans, and R.~Jensen.
	\newblock The infinity {L}aplacian, {A}ronsson's equation and their
	generalizations.
	\newblock {\em Trans. Amer. Math. Soc.}, 360(1):77--101, 2008.
	
	\bibitem{CaCa2016}
	S.~Cacace and F.~Camilli.
	\newblock Ergodic problems for {H}amilton-{J}acobi equations: yet another but
	efficient numerical method.
	\newblock 01 2016.
	
	\bibitem{CGMT}
	F.~Cagnetti, D.~Gomes, H.~Mitake, and H.~V. Tran.
	\newblock A new method for large time behavior of degenerate viscous
	{H}amilton-{J}acobi equations with convex {H}amiltonians.
	\newblock {\em Ann. Inst. H. Poincar\'e Anal. Non Lin\'eaire}, 32(1):183--200,
	2015.
	
	\bibitem{camillinet}
	F.~Camilli, E.~Carlini, and C~Marchi.
	\newblock A model problem for mean field games on networks.
	\newblock {\em Discrete {C}ontin. {D}yn. {S}yst.}, 35(9):4173--4192, 2015.
	
	\bibitem{MR3146865}
	F.~Camilli, A.~Festa, and D.~Schieborn.
	\newblock An approximation scheme for a {H}amilton-{J}acobi equation defined on
	a network.
	\newblock {\em Appl. Numer. Math.}, 73:33--47, 2013.
	
	\bibitem{Cd2}
	P.~Cardaliaguet.
	\newblock Weak solutions for first order mean-field games with local coupling.
	\newblock {\em Preprint}, 2013.
	
	\bibitem{GraCard}
	P.~Cardaliaguet and P.~J. Graber.
	\newblock Mean field games systems of first order.
	\newblock {\em ESAIM Control Optim. Calc. Var.}, 21(3):690--722, 2015.
	
	\bibitem{ChCoSm77}
	K.~N. Chueh, C.~C. Conley, and J.~A. Smoller.
	\newblock Positively invariant regions for systems of nonlinear diffusion
	equations.
	\newblock {\em Indiana Univ. Math. J.}, 26(2):373--392, 1977.
	
	\bibitem{Ci}
	M.~Cirant.
	\newblock Nonlinear pdes in ergodic control, mean-field games and prescribed
	curvature problems.
	\newblock {\em Thesis}, 2013.
	
	\bibitem{MR2237158}
	Andrea Davini and Antonio Siconolfi.
	\newblock A generalized dynamical approach to the large time behavior of
	solutions of {H}amilton-{J}acobi equations.
	\newblock {\em SIAM J. Math. Anal.}, 38(2):478--502 (electronic), 2006.
	
	\bibitem{DeStTz00}
	Sophia Demoulini, David M.~A. Stuart, and Athanasios~E. Tzavaras.
	\newblock Construction of entropy solutions for one-dimensional elastodynamics
	via time discretisation.
	\newblock {\em Ann. Inst. H. Poincar\'e Anal. Non Lin\'eaire}, 17(6):711--731,
	2000.
	
	\bibitem{DiPerna83}
	R.~J. DiPerna.
	\newblock Convergence of approximate solutions to conservation laws.
	\newblock {\em Arch. Rational Mech. Anal.}, 82(1):27--70, 1983.
	
	\bibitem{FATH4}
	A.~Fathi.
	\newblock Sur la convergence du semi-groupe de {L}ax-{O}leinik.
	\newblock {\em C. R. Acad. Sci. Paris S\'er. I Math.}, 327:267--270, 1998.
	
	\bibitem{FG2}
	R.~Ferreira and D.~Gomes.
	\newblock Existence of weak solutions for stationary mean-field games through
	variational inequalities.
	\newblock {\em Preprint}.
	
	\bibitem{GMit}
	D.~Gomes and H.~Mitake.
	\newblock Existence for stationary mean-field games with congestion and
	quadratic {H}amiltonians.
	\newblock {\em NoDEA Nonlinear Differential Equations Appl.}, 22(6):1897--1910,
	2015.
	
	\bibitem{GP}
	D.~Gomes, L.~Nurbekyan, and M.~Prazeres.
	\newblock Explicit solutions of one-dimensional first-order stationary
	mean-field games with a generic nonlinearity.
	\newblock {\em Preprint}, 2016.
	
	\bibitem{GLP2}
	D.~Gomes, L.~Nurbekyan, and M.~Prazeres.
	\newblock Explicit solutions of one-dimensional first-order stationary
	mean-field games with congestion.
	\newblock {\em Preprint}, 2016.
	
	\bibitem{GPat}
	D.~Gomes and S.~Patrizi.
	\newblock Obstacle mean-field game problem.
	\newblock {\em Interfaces Free Bound.}, 17(1):55--68, 2015.
	
	\bibitem{GPatVrt}
	D.~Gomes, S.~Patrizi, and V.~Voskanyan.
	\newblock On the existence of classical solutions for stationary extended mean
	field games.
	\newblock {\em Nonlinear Anal.}, 99:49--79, 2014.
	
	\bibitem{GPim2}
	D.~Gomes and E.~Pimentel.
	\newblock Time dependent mean-field games with logarithmic nonlinearities.
	\newblock {\em To appear in SIAM Journal on Mathematical Analysis}.
	
	\bibitem{GPim1}
	D.~Gomes and E.~Pimentel.
	\newblock Local regularity for mean-field games in the whole space.
	\newblock {\em To appear in Minimax Theory and its Applications}, 2015.
	
	\bibitem{GPff}
	D.~Gomes and E.. Pimentel.
	\newblock Regularity for mean-field games systems with initial-initial boundary
	conditions: subquadratic case.
	\newblock {\em Preprint}, 2015.
	
	\bibitem{MR2396521}
	Hitoshi Ishii.
	\newblock Asymptotic solutions for large time of {H}amilton-{J}acobi equations
	in {E}uclidean {$n$} space.
	\newblock {\em Ann. Inst. H. Poincar\'e Anal. Non Lin\'eaire}, 25(2):231--266,
	2008.
	
	\bibitem{ll1}
	J.-M. Lasry and P.-L. Lions.
	\newblock Jeux \`a champ moyen. {I}. {L}e cas stationnaire.
	\newblock {\em C. R. Math. Acad. Sci. Paris}, 343(9):619--625, 2006.
	
	\bibitem{ll2}
	J.-M. Lasry and P.-L. Lions.
	\newblock Jeux \`a champ moyen. {II}. {H}orizon fini et contr\^ole optimal.
	\newblock {\em C. R. Math. Acad. Sci. Paris}, 343(10):679--684, 2006.
	
	\bibitem{llg2}
	J.-M. Lasry, P.-L. Lions, and O.~Gu{\'e}ant.
	\newblock Mean field games and applications.
	\newblock {\em Paris-Princeton lectures on Mathematical Finance}, 2010.
	
	\bibitem{MR1457088}
	Gawtum Namah and Jean-Michel Roquejoffre.
	\newblock Comportement asymptotique des solutions d'une classe d'\'equations
	paraboliques et de {H}amilton-{J}acobi.
	\newblock {\em C. R. Acad. Sci. Paris S\'er. I Math.}, 324(12):1367--1370,
	1997.
	
	\bibitem{porretta2}
	A.~Porretta.
	\newblock Weak solutions to {F}okker-{P}lanck equations and mean field games.
	\newblock {\em Arch. Ration. Mech. Anal.}, 216(1):1--62, 2015.
	
\end{thebibliography}
\end{document}